\documentclass[11pt]{amsart}
\usepackage{mathrsfs,xcolor}

\title[On the identities of 2-dimensional tr-prime algebras]{On the trace identities of 2-dimensional tr-prime algebras over a finite field}
\author[L.~Centrone et al]{Lucio Centrone}
\address{Dipartimento di Matematica, Universit\`a degli Studi di Bari Aldo Moro, Via Edoardo Orabona, 4, 70125 Bari, Italy}
\email{lucio.centrone@uniba.it}
\thanks{L. Centrone was partially supported by PNRR-MUR PE0000023-NQSTI}

\author[]{Rosiele Trindade Barbosa}
\address{Department of Mathematics, Instituto de Matem\'atica, Estat\'istica e Ci\^encia da Computa\c c\~ao, Universidade de S\~ao Paulo, SP, Brazil}
\email{rosiele@ime.usp.br}
\thanks{R.~Barbosa was financed in part by the Coordena\c c\~ao de Aperfei\c coamento de Pessoal de N\'ivel Superior - Brasil (CAPES) - Finance Code 001.}

\author[]{Felipe Yukihide Yasumura}
\address{Department of Mathematics, Instituto de Matem\'atica, Estat\'istica e Computa\c c\~ao Cient\'ifica, Universidade Estadual de Campinas, SP, Brazil}
\email{yukihide@unicamp.br}
\thanks{F.~Yasumura was supported by Fapesp grant no.~2024/14914-9 and 2025/16296-3.}

\newtheorem{thm}{Theorem}
\newtheorem{lemma}[thm]{Lemma}
\newtheorem{proposition}[thm]{Proposition}

\theoremstyle{definition}
\newtheorem{definition}[thm]{Definition}
\theoremstyle{remark}

\newtheorem{remark}[thm]{Remark}

\renewcommand{\labelenumi}{(\roman{enumi})}

\begin{document}
\begin{abstract}
We study associative algebras endowed with a trace map. We classify the two-dimensional trace-prime algebras over an arbitrary field. In addition, when the base field is finite, for each such algebra, we provide a finite basis for the T-ideal of trace polynomial identities.
\end{abstract}
\maketitle

\section{Introduction}

The Specht problem asks whether the set of all polynomial identities satisfied by an algebra admits a finite basis. It is well known that this problem has a positive solution for associative algebras over fields of characteristic zero \cite{Kemer} (see also \cite{AKBK}). Positive solutions have also been obtained for several important varieties of nonassociative algebras \cite{ilt_alt,vais_zelm}, as well as for varieties generated by finite algebras (including finite groups); see \cite{BO,Kruse,Lvov,Lvov2,Medvedev1,Medvedev2,OaPo,Polin}.

Despite these general existence results, determining an explicit finite basis for the polynomial identities of a given algebra is, in general, a difficult task. A significant part of the development of PI-theory has therefore been devoted to finding explicit finite bases for important classes of algebras.

For instance, over a field of characteristic zero, the polynomial identities of finite-dimensional central simple associative algebras are known only when the dimension is at most $4$. This was first established by Razmyslov \cite{Raz73} and later refined by Drensky \cite{Drensky81}. Over an infinite field of positive odd characteristic, analogous results were obtained by Koshlukov \cite{Ko01}. On the other hand, when the base field is finite, a finite basis for the polynomial identities of the matrix algebra $M_n(\mathbb{F})$ is known only for $n\leq 4$ (see \cite{Gen81,GenSid82,MalKuz78}).

In the context of associative polynomial identities with trace, Razmyslov \cite{Raz74} and Procesi \cite{Pro}, independently, described the polynomial identities of the algebra of $n\times n$ matrices endowed with the trace map. Moreover, the trace of the matrix algebra is closely related to the existence of the so-called central polynomials. In general, the cocharacter sequence of trace algebras was studied by Berele and Regev \cite{BR}. In particular, the authors applied the results to study the algebras of diagonal matrices. In addition, in \cite{AB}, Berele determined the $T$-ideal of trace identities of the algebra of diagonal matrices and upper triangular matrices. The T-ideal of polynomial identities of the algebra of diagonal matrices endowed with other trace maps was investigated in \cite{IKM21,IKM22} (see \cite{IKM} as well).

In this paper, in order to obtain a finite partial analogue of these results, we study the polynomial identities of 2-dimensional tr-prime associative algebras. First, we classify these algebras (Proposition~\ref{2_dim}). Then, our main result is an explicit finite basis for the trace polynomial identities of each  such algebra when the base field is finite (Theorems~\ref{thm_alphabeta}, \ref{thm_lambda}, \ref{thmA}, \ref{thmB} and \ref{thmC}). Finally, in Section~\ref{sec}, we discuss and compare our approach with the classical approach to studying polynomial identities of finite algebras.

\section{Preliminaries}
In this section, unless otherwise stated, we assume that $\mathbb{F}$ is an arbitrary field.

\subsection{Trace algebras}
Let $\mathcal{A}$ be an associative algebra over a field $\mathbb{F}$. A \emph{trace} on $\mathcal{A}$ is an $\mathbb{F}$-linear map $\mathrm{tr}:\mathcal{A}\to\mathcal{A}$ such that, for all $a$, $b\in\mathcal{A}$:
\begin{enumerate}
\item $\mathrm{tr}(a)\in Z(\mathcal{A})$,
\item $\mathrm{tr}(ab)=\mathrm{tr}(ba)$,
\item $\mathrm{tr}(\mathrm{tr}(a)b)=\mathrm{tr}(a)\mathrm{tr}(b)$.
\end{enumerate}
A \emph{trace algebra}, or simply a \emph{tr-algebra}, is an algebra endowed with a trace map. Often, a trace map is defined to be a linear map $\mathcal{A}\to\mathbb{F}$, which is a particular case of our definition when $\mathcal{A}$ is unital.

A homomorphism of algebras with trace is an algebra homomorphism $\psi:(\mathcal{A},t)\to(\mathcal{A}',t')$ such that $t'(\psi(a))=\psi(t(a))$, for all $a\in\mathcal{A}$. If $\psi$ is bijective, then we say that $(\mathcal{A},t)$ and $(\mathcal{A}',t')$ are \emph{isomorphic}, denoted $(\mathcal{A},t)\cong(\mathcal{A}',t')$.

A \emph{tr-ideal} is an ideal $I\subseteq\mathcal{A}$ such that $t(a)\in I$, for all $a\in I$. We say that $\mathcal{A}$ is:
\begin{enumerate}
\item \emph{tr-simple} if $\mathcal{A}^2\ne0$ and the only tr-ideals of $\mathcal{A}$ are $0$ and $\mathcal{A}$;
\item \emph{tr-prime} if for any pair of nonzero tr-ideals $I$, $J\subseteq\mathcal{A}$, we have $IJ\ne0$;
\item \emph{tr-semiprime} if for any nonzero tr-ideal $I\subseteq\mathcal{A}$, we have $I^2\ne0$.
\end{enumerate}

\subsection{$\Omega$-algebras}
We let $\Omega=\bigcup_{n\ge0}\Omega_n$. An $\Omega$-algebra is an $\mathbb{F}$-vector space $\mathcal{A}$ such that for any $\omega\in\Omega_n$, we have an $n$-linear operation $\omega:\mathcal{A}\times\cdots\times\mathcal{A}\to\mathcal{A}$. The $0$-ary operations are a choice of a distinguished element in $\mathcal{A}$.

Given a set $X$, we build $M(X)$ inductively via:
\begin{itemize}
\item $X\subseteq M(X)$, and
\item for any $\omega\in\Omega_n$ ($n\ge0$), and $w_1$, \ldots, $w_n\in M(X)$, then one has $\omega(w_1,\ldots,w_n)\in M(X)$.
\end{itemize}
Then, the $\mathbb{F}$-vector space with basis $M(X)$ has a natural structure of an $\Omega$-algebra. It is the absolute free $\Omega$-algebra over $\mathbb{F}$.

\subsection{Free trace algebras}
We consider the absolute free $\Omega$-algebra, where $\Omega=\Omega_0\cup\Omega_1\cup\Omega_2$, and $\Omega_0=\{1\}$, $\Omega_1=\{t\}$, and $\Omega_2=\{\cdot\}$. The variety of unital associative algebras with trace is the variety of $\Omega$-algebras defined by the following polynomial identities:
\begin{align*}
&(x\cdot y)\cdot z-x\cdot(y\cdot z),\quad 1\cdot x-x,\quad x\cdot 1-x,\\%
&t(x)\cdot y-y\cdot t(x),\quad t(x\cdot y)-t(y\cdot x),\quad t(t(x)\cdot y)-t(x)\cdot t(y).
\end{align*}
Its free algebra, freely generated by a set $X$, is denoted by $\mathbb{F}\langle X,t\rangle$. As customary, the operation $\cdot$ is represented via juxtaposition of monomials.

The set of all polynomial identities of a trace algebra $(\mathcal{A},t)$ in $\mathbb{F}\langle X,t\rangle$ is denoted by $\mathrm{Id}(\mathcal{A},t)$. Given a subset $S\subseteq\mathbb{F}\langle X,t\rangle$, $\langle S\rangle$ denotes the smallest T-ideal containing $S$.

\subsection{Central closure}
The following construction can be carried out for prime $\Omega$-algebras (see \cite[\S3]{Razmyslov}). Let $(\mathcal{A},t)$ be a tr-prime algebra. Let $\mathrm{M}(\mathcal{A})$ be the subalgebra of $\mathrm{End}_\mathbb{F}(\mathcal{A})$ generated by all left and right multiplications and the linear map $t$. Then, there is a pair of algebras, namely, the unital associative and commutative \emph{centroid} $C(\mathcal{A})$ and the \emph{central closure} $Q(\mathcal{A})$, characterized by the following properties (see \cite[Proposition 3.2]{Razmyslov}):
\begin{enumerate}
\item $Q(\mathcal{A})$ is a $C(\mathcal{A})$-algebra containing $\mathcal{A}$, and $Q(\mathcal{A})=C(\mathcal{A})\mathcal{A}$;
\item Every nonzero $\mathrm{M}(\mathcal{A})$-submodule of $Q(\mathcal{A})$ has nonzero intersection with $\mathcal{A}$;
\item For every homomorphism of $\mathrm{M}(\mathcal{A})$-modules $\chi:I\to\mathcal{A}$, where $I\subseteq\mathcal{A}$ is an ideal, there exists a unique $c\in C(\mathcal{A})$ such that $\chi(a)=ca$, for all $a\in I$.
\end{enumerate}
In addition, $C(\mathcal{A})$ is a field and $Q(\mathcal{A})$ is tr-prime.

Connected to the construction of the pair, we have:
\begin{lemma}[particular case of {\cite[Theorem 4.1]{Razmyslov}}]\label{lemRaz3}
Let $\mathcal{A}$ be a tr-prime algebra, and let $\mathcal{V}\subseteq\mathcal{A}$ be a subspace. Then, $\dim_{C(\mathcal{A})}C(\mathcal{A})\mathcal{V}<m$ if and only if for every polynomial $f=f(x_1,\ldots,x_m,y_1,\ldots,y_r)$ that is multilinear and alternating in $\{x_1,\ldots,x_m\}$, then $f(v_1,\ldots,v_m,a_1,\ldots,a_r)=0$, for all $v_1,\ldots,v_m\in\mathcal{V}$ and for all $a_1,\ldots,a_r\in\mathcal{A}$.
\end{lemma}
A multilinear polynomial that is alternating in a set of $n$ variables is called a \emph{Capelli polynomial of order $n$}.

\subsection{Further results on finite prime algebras}
Embeddings of finite prime algebras are determined by embeddings of the respective T-ideal of polynomial identities:
\begin{lemma}[{\cite[part of Theorem~8]{BDDF}}]\label{section}
Let $\mathcal{A}\in\mathrm{var}(\mathcal{A}')$, where both $\mathcal{A}$ and $\mathcal{A}'$ are finite tr-algebras and $\mathcal{A}$ is tr-prime. Then $\mathcal{A}$ is a section of $\mathcal{A}'$, that is, there exists a subalgebra $\mathcal{S}\subseteq\mathcal{A}'$ and a surjective algebra homomorphism $\mathcal{S}\to\mathcal{A}$.
\end{lemma}

As a consequence, we obtain:
\begin{lemma}[{\cite[Corollary~9]{BDDF}}]\label{iso_fin}
Let $\mathcal{A}$ and $\mathcal{A}'$ be finite tr-prime algebras satisfying the same set of polynomial identities. Then $\mathcal{A}\cong\mathcal{A}'$.
\end{lemma}

\section{Prime tr-algebras}

In this section, $\mathbb{F}$ denotes an arbitrary field. We shall present some families of tr-prime algebras, and we classify the $2$-dimensional tr-prime algebras.

\subsection{Diagonal matrices}

Let $\alpha_1,\ldots,\alpha_n\in\mathbb{F}$. Then, $(D_n,t_{(\alpha_1,\ldots,\alpha_n)})$ denotes the algebra $\mathbb{F}^n$ endowed with the trace
\[
(a_1,\ldots,a_n)\mapsto(\alpha_1a_1+\cdots+\alpha_na_n)(1,\ldots,1).
\]

The notation $D_n(\mathbb{F})$ will be used when we need to emphasize that the underlying algebra structure is $\mathbb{F}^n$. It is clear that any trace map of the kind $D_n\to\mathbb{F}$ is equal to some $t_{(\alpha_1,\ldots,\alpha_n)}$. In addition, it is elementary to prove that $(D_n,t_{(\alpha_1,\ldots,\alpha_n)})\cong(D_m,t_{(\beta_1,\ldots,\beta_m)})$ if and only if $n=m$ and there exists $\sigma\in\mathcal{S}_n$ such that $\alpha_i=\beta_{\sigma(i)}$, for each $i\in\{1,2,\ldots,n\}$.

\begin{definition}
A trace $t$ on an algebra $\mathcal{A}$ is called \emph{degenerate} if there exists a non-zero element $a \in \mathcal{A}$ such that, for all $b\in \mathcal{A}$, one has $t(ab)=0$. Otherwise, we say that $t$ is \emph{nondegenerate}.
\end{definition}

Now, we have a complement to \cite[Lemma 6]{IKM22}:
\begin{lemma}\label{unnamed}
Assume that $n\ge2$. The following conditions on $(D_n,t_{(\alpha_1,\ldots,\alpha_n)})$ are equivalent:
\begin{enumerate}
\item $t_{(\alpha_1,\ldots,\alpha_n)}$ is nondegenerate,
\item $\alpha_i\ne0$, for all $i\in\{1,2,\ldots,n\}$,
\item $(D_n,t_{(\alpha_1,\ldots,\alpha_n)})$ is tr-simple.
\end{enumerate}
\end{lemma}
\begin{proof}
(i)$\iff$(ii): Proved in \cite[Lemma 6]{IKM22}.

(ii)$\iff$(iii): The minimal ordinary ideals of $D_n$ are of the form $\mathbb{F}e_i$, where $e_i\in D_n$ has its $i$-th entry $1$ and $0$ elsewhere. If $\alpha_i=0$ for some $i$, then $\mathbb{F}e_i$ is a nonzero proper tr-ideal of $D_n$, which implies that $D_n$ is not tr-simple. Conversely, assume that $D_n$ is not tr-simple and let $I\subseteq D_n$ be a nonzero proper tr-ideal. Then, $I=\mathbb{F}e_{i_1}+\cdots+\mathbb{F}e_{i_k}$, for some $i_1$, \dots, $i_k\in\{1,\ldots,n\}$. Since $I$ is a tr-ideal, one gets $\mathrm{tr}(e_{i_\ell})=\alpha_{i_\ell}\in I$, which implies $\alpha_{i_\ell}=0$. Therefore, $\alpha_i=0$ for some $i$.
\end{proof}

Finally, we shall classify the trace structures on $\mathbb{F}\oplus\mathbb{F}$.

Given $\mu_1,\mu_2,\nu_1,\nu_2\in\mathbb{F}$, the quadruple $\mathcal{A}(\mu_1,\mu_2,\nu_1,\nu_2)$ denotes the algebra $\mathbb{F}\oplus\mathbb{F}$, with basis $\{e_1,e_2\}$ (where $e_ie_j=\delta_{ij}e_i$, for all $i$, $j$), and with a unary operation defined by $t(e_k)=\mu_ke_1+\nu_ke_2$, $k\in\{1,2\}$. Note that $\mathcal{A}(\mu_1,\mu_2,\nu_1,\nu_2)\cong\mathcal{A}(\nu_1,\nu_2,\mu_1,\mu_2)$ and $(D_2,t_{(\alpha,\beta)})\cong\mathcal{A}(\alpha,\beta,\alpha,\beta)$.

It is clear that any tr-algebra having $\mathbb{F}\oplus\mathbb{F}$ as underlying algebra structure is isomorphic to some $\mathcal{A}(\mu_1,\mu_2,\nu_1,\nu_2)$. However, not every $\mathcal{A}(\mu_1,\mu_2,\nu_1,\nu_2)$ is a trace algebra. For instance, $\mathcal{A}(\alpha,\beta,0,\alpha+\beta)$ does not necessarily satisfy $t(t(e_1))=t(e_1)t(1)$.

\begin{lemma}\label{rem}
The algebra $\mathcal{A}(\mu_1,\mu_2,\nu_1,\nu_2)$ is a trace algebra if and only if $\mu_1=\nu_1$ and $\mu_2=\nu_2$, or $\mu_2=\nu_1=0$. In addition, $\mathcal{A}$ is tr-prime if and only if $\mu_1=\nu_1$, $\mu_2=\nu_2$ and $(\mu_1,\mu_2)\ne(0,0)$.
\end{lemma}
\begin{proof}
It is enough to analyze all the relations obtained from the equality $t(t(e_i)e_j)=t(e_i)t(e_j)$.
\end{proof}

\subsection{Quadratic extension of the base field}
We assume that $\mathbb{E}$ is a quadratic extension of the base field $\mathbb{F}$. We classify the structure of trace $\mathbb{F}$-algebras on $\mathbb{E}$.

\begin{lemma}\label{tr_simple}
Let $t:\mathbb{E}\to\mathbb{E}$ be a trace map of the $\mathbb{F}$-algebra $\mathbb{E}$. Then either:
\begin{enumerate}
\renewcommand{\labelenumi}{(\roman{enumi})}
\item there exists $\alpha\in\mathbb{E}$ such that $t(x)=\alpha x$, for all $x\in\mathbb{E}$, or
\item there exists $c\in\mathbb{E}\setminus\{0\}$ such that $t(x)=\mathrm{Tr}_{\mathbb{E}/\mathbb{F}}(cx)$, for all $x\in\mathbb{E}$.%$\mathrm{Im}\,t=\mathbb{F}$.
\end{enumerate}
\end{lemma}
\begin{proof}
If $t=0$, then $t$ is of the kind (i) with $\alpha=0$. So, we shall assume that $t\ne0$. Suppose that $t(\mathbb{E})=\mathbb{E}$. Then, for any $x\in\mathbb{E}$, we have $x=t(y)$, for some $y\in\mathbb{E}$. Hence, $t(x)=t(t(y)\cdot1)=t(y)t(1)=\alpha x$, where $\alpha=t(1)$.

Now, suppose that $t(\mathbb{E})=\mathbb{F}\cdot t(u)$, for some $u\in\mathbb{E}$. Then, $\mathbb{F}\cdot t(u)\ni t(ut(u))=t(u)t(u)$, that is, $t(u)t(u)=\lambda_0t(u)$, for some $\lambda_0\in\mathbb{F}$. Since $t(u)\ne0$, we get $t(u)=\lambda_0\in\mathbb{F}$. Hence, $t(\mathbb{E})=\mathbb{F}$.

Now, let $x_0\in\mathbb{E}$ be such that $t(x_0)=1$, and let $\mathrm{Ker}\,t=\mathbb{F}y_0$. Similarly, write $\mathrm{Ker}\,\mathrm{Tr}_{\mathbb{E}/\mathbb{F}}=\mathbb{F}y$, for some $y\in\mathbb{E}$. Note that $\mathrm{Tr}_{\mathbb{E}/\mathbb{F}}(yy_0^{-1}x_0)=0$ implies that $yy_0^{-1}x_0\in\mathrm{Ker}\,\mathrm{Tr}_{\mathbb{E}/\mathbb{F}}=\mathbb{F}y$. Consequently, $y_0^{-1}x_0\in\mathbb{F}$, which would imply that $\{x_0,y_0\}$ is a linearly dependent set, a contradiction. Thus, we may define
\[
c=\mathrm{Tr}_{\mathbb{E}/\mathbb{F}}(yy_0^{-1}x_0)^{-1}yy_0^{-1}.
\]
It is then clear that $t(x)=\mathrm{Tr}_{\mathbb{E}/\mathbb{F}}(cx)$, for all $x\in\mathbb{E}$.
\end{proof}

\begin{remark}
Denote by $t_{\alpha_0}$ the trace map given by $t_{\alpha_0}(x)=\alpha_0x$. It is clear that $(\mathbb{E},t_{\alpha_0})=(\mathbb{E},t_{\alpha_0'})$ if and only if $\alpha_0=\alpha_0'$. In addition, $(\mathbb{E},t_{\alpha_0})\cong(\mathbb{E},t_{\alpha_0'})$ if and only if $\alpha_0'\in\{\varphi(\alpha_0)\mid\varphi\in\mathrm{Aut}(\mathbb{E}/\mathbb{F})\}$.

Denote by $\mathrm{Tr}_c$ the trace map given by $\mathrm{Tr}_c(x)=\mathrm{Tr}_{\mathbb{E}/\mathbb{F}}(cx)$. It is clear that $(\mathbb{E},\mathrm{Tr}_c)=(\mathbb{E},\mathrm{Tr}_{c'})$ if and only if $c=c'$. In addition, $(\mathbb{E},\mathrm{Tr}_c)\cong(\mathbb{E},\mathrm{Tr}_{c'})$ if and only if $c'\in\{\varphi(c)\mid\varphi\in\mathrm{Aut}(\mathbb{E}/\mathbb{F})\}$.

%Let $t_0:\mathbb{E}\to\mathbb{F}$ and $t_0':\mathbb{E}\to\mathbb{F}$ be trace maps, having $\mathbb{F}$ as their image. Then, it is clear that $(\mathbb{E},t_0)\cong(\mathbb{E},t_0')$ if and only if $t_0'\in t_0\circ\mathrm{Aut}(\mathbb{E}/\mathbb{F})$.
\end{remark}

Now, we have:
\begin{lemma}
Let $t$ be a trace map on $\mathbb{E}$ and denote $\mathcal{A}=(\mathbb{E},t)$.
\begin{enumerate}
\item If $t=t_\alpha$, then $\dim_{C(\mathcal{A})}Q(\mathcal{A})=1$,
\item If $t=\mathrm{Tr}_c$, then $\dim_{C(\mathcal{A})}Q(\mathcal{A})=2$.
\end{enumerate}
\end{lemma}
\begin{proof}
\noindent(i) Note that, in this case, $t$ is an $\mathbb{E}$-linear map. Hence, $\mathbb{E}\subseteq C(\mathcal{A})$. Thus, $\dim_{C(\mathcal{A})}Q(\mathcal{A})\le1$.

\noindent(ii) Since $\dim_\mathbb{F}\mathcal{A}=2$, we obtain $\dim_{C(\mathcal{A})}Q(\mathcal{A})\le2$. The $\mathbb{F}$-linear map $t:\mathcal{A}\to\mathbb{F}$ has nontrivial kernel. Let $0\ne\alpha\in\mathrm{Ker}\,t$ and $\beta\not\in\mathrm{Ker}\,t$. The polynomial  $f(x_1,x_2,x_3):=t(x_3x_1)x_2-t(x_3x_2)x_1$ is Capelli of order $2$. It is not a polynomial identity of $\mathcal{A}$, since $f(\beta\alpha^{-1},1,\alpha)=t(\beta)\ne0$. Thus, from Lemma~\ref{lemRaz3}, we obtain $\dim_{C(\mathcal{A})}Q(\mathcal{A})\ge2$.
\end{proof}

\subsection{$2$-dimensional nonsemiprime tr-prime algebra}
Let $\mathcal{A}=\mathbb{F}[\xi]/\langle\xi^2\rangle$, and let $v=\xi+\langle\xi^2\rangle\in\mathcal{A}$. Given $\lambda\in\mathbb{F}$, we define the $\mathbb{F}$-linear map $t_\lambda:\mathcal{A}\to\mathbb{F}$ via $t_\lambda(1)=\lambda$ and $t_\lambda(v)=1$. This trace algebra is denoted by $(\mathbb{F}\cdot1+\mathbb{F}v,\lambda)$. It is clear that $(\mathbb{F}\cdot1+\mathbb{F}v,\lambda)\cong(\mathbb{F}\cdot1+\mathbb{F}v,\lambda')$ if and only if $\lambda=\lambda'$.

\begin{lemma}\label{trprime_nonsemiprime}
Let $\mathcal{A}$ be a $2$-dimensional tr-prime algebra that is not semiprime as an ordinary algebra. Then $\mathcal{A}\cong(\mathbb{F}\cdot1+\mathbb{F}v,\lambda)$, for some $\lambda\in\mathbb{F}$.
\end{lemma}
\begin{proof}
Since $\mathcal{A}$ is not semiprime as an ordinary algebra, it has a nontrivial Jacobson radical. If $J(\mathcal{A})=\mathcal{A}$, then $\mathcal{A}$ is nilpotent, which contradicts the tr-primality of $\mathcal{A}$. Hence, $\dim J(\mathcal{A})=1$. In particular, $\mathcal{A}/J(\mathcal{A})\cong\mathbb{F}$. By Wedderburn-Malcev Theorem and since $\mathcal{A}$ is unital, we obtain $\mathcal{A}\cong\mathbb{F}\cdot1+\mathbb{F}v$, where $v^2=0$.

First, assume that there exists $a\in\mathcal{A}$ with $t(a)\notin\mathbb{F}$. Then $\{1,t(a)\}$ is a $\mathbb{F}$-basis of $\mathcal{A}$. In particular, each $x\in\mathcal{A}$ can be written as a linear combination of $1$ and $t(v)$. Since $t$ is $\mathbb{F}$-linear and trace-linear, we get $t(x)=xt(1)$. In particular, $t(v)=vt(1)\in\mathbb{F}v$, since $\mathbb{F}v$ is an ideal. Hence, $\mathbb{F}v$ is a nilpotent tr-ideal, contradicting the tr-primality of $\mathcal{A}$. Thus, the range of the trace map is contained in $\mathbb{F}$.

Now, if $t(v)=0$, then $\mathbb{F}v$ is a nilpotent tr-ideal, once again contradicting the tr-primality of $\mathcal{A}$. Hence, $t(v)=\gamma\in\mathbb{F}\setminus\{0\}$. Rescaling, we may assume that $t(v)=1$. Therefore, we obtain $\mathcal{A}\cong(\mathbb{F}\cdot1+\mathbb{F}v,t(1))$.
\end{proof}

\subsection{Classification of two-dimensional tr-prime algebras}
Recall that we assume that $\mathbb{F}$ is an arbitrary field in this section. We classify $2$-dimensional tr-prime algebras over $\mathbb{F}$.

\begin{proposition}\label{2_dim}
A unital $2$-dimensional tr-prime algebra over an arbitrary field $\mathbb{F}$ is isomorphic to one of the following:
\begin{enumerate}
\item $(D_2,t_{(\alpha,\beta)})$, where $\alpha$, $\beta\in\mathbb{F}$ and $(\alpha,\beta)\ne(0,0)$, $D_2=\mathbb{F}\oplus\mathbb{F}$ and $t_{(\alpha,\beta)}(a,b)=\alpha a+\beta b$,
\item A quadratic extension $\mathbb{E}$ of $\mathbb{F}$, endowed with the trace $t_{\alpha_0}(x)=\alpha_0x$, for some $\alpha_0\in\mathbb{E}$,
\item A quadratic extension $\mathbb{E}$ of $\mathbb{F}$, endowed with the trace $\mathrm{Tr}_c(x)=\mathrm{Tr}_{\mathbb{E}/\mathbb{F}}(cx)$, for some $c\in\mathbb{E}\setminus\{0\}$,
\item $(\mathbb{F}\cdot1+\mathbb{F}v,\lambda)$, where $\lambda\in\mathbb{F}$, $v^2=0$ and trace $t_\lambda(a\cdot1+bv)=\lambda a +b$.
\end{enumerate}
\end{proposition}
\begin{proof}
Note that $\mathcal{A}$ cannot be nilpotent. If $J(\mathcal{A})=0$, then the underlying algebra is either $\mathcal{A}=\mathbb{F}\oplus\mathbb{F}$ or a quadratic extension of $\mathbb{F}$. Then, $\mathcal{A}$ is either isomorphic to $(D_2,t_{(\alpha,\beta)})$ for some $\alpha$, $\beta\in\mathbb{F}$ with $(\alpha,\beta)\ne(0,0)$ (Lemma~\ref{rem}), or $\mathcal{A}$ is isomorphic to one of the algebras of the kind (ii) or (iii) (Lemma~\ref{tr_simple}).

If $J(\mathcal{A})\ne0$, then Lemma~\ref{trprime_nonsemiprime} gives that $\mathcal{A}$ is isomorphic to the algebra of the kind (iv).
\end{proof}

As an application, we obtain the following results.

\begin{lemma}\label{Cap_3}
Assume that a tr-prime $\mathbb{F}$-algebra $\mathcal{A}$ satisfies $x^q-x$ and all Capelli polynomials of order $3$, where $|\mathbb{F}|=q$. Then, $\mathcal{A}$ is isomorphic to an $\mathbb{F}$-subalgebra of $D_2(\mathbb{F})$.
\end{lemma}
\begin{proof}
Since $\mathcal{A}$ satisfies all Capelli polynomials of order $3$, Razmyslov Rank Theorem (Lemma~\ref{lemRaz3}) implies that $\dim_{C(\mathcal{A})}Q(\mathcal{A})\leq2$. The case $\dim_{C(\mathcal{A})}Q(\mathcal{A})=1$ is trivial; hence, let us suppose $\dim_{C(\mathcal{A})}Q(\mathcal{A})=2$. Once $\mathcal{A}$ satisfies $x^q-x$, then it should be a subring of the fixed ring $\{x\in Q(\mathcal{A})\mid x^q=x\}$. Since $Q(\mathcal{A})$ is tr-prime, then it is one of the algebras listed in Proposition~\ref{2_dim}. Then, the possibilities of the fixed subring are either $\mathbb{F}$ or $\mathbb{F}\oplus\mathbb{F}$. The result is proved.
\end{proof}

The trace map is always linear with respect to the base field. If we allow a smaller base field, then we obtain a mild variation of the previous lemma:
\begin{lemma}\label{Cap_3v}
Assume that a tr-prime $\mathbb{F}$-algebra $\mathcal{A}$ satisfies $x^{q^2}-x$ and all Capelli polynomials of order $3$, where $|\mathbb{F}|=q$. Then, $\mathcal{A}$ is isomorphic to an $\mathbb{F}$-subalgebra of $D_2(\mathbb{E})$ or an $\mathbb{F}$-subalgebra of $\mathbb{E}$, where $\mathbb{E}/\mathbb{F}$ is a quadratic extension and $\mathbb{E}$ has one of the trace structures of Lemma~\ref{tr_simple}.
\end{lemma}
\begin{proof}
As in the proof of Lemma~\ref{Cap_3}, we obtain that $Q(\mathcal{A})$ is one of the algebras listed in Proposition~\ref{2_dim}: $D_2(C(\mathcal{A}))$, a quadratic extension of $C(\mathcal{A})$ as a $C(\mathcal{A})$-algebra, or $C(\mathcal{A})\cdot1+C(\mathcal{A})v$. The fixed algebra $\{x\in Q(\mathcal{A})\mid x^{q^2}=x\}$ is contained in $\mathbb{E}$, in the first and third cases, and is contained in $D_2(\mathbb{E})$ in the second case. The structure of trace $\mathbb{F}$-algebras on $\mathbb{E}$ is described in Lemma~\ref{tr_simple}.
\end{proof}

\begin{lemma}\label{Cap_3_weak}
Assume that a tr-prime $\mathbb{F}$-algebra $(\mathcal{A},t)$ satisfies $(x^q-x)^2$, $t(x)^q-t(x)$ and all Capelli polynomials of order $3$, where $q=|\mathbb{F}|$. If $\mathcal{A}$ does not satisfy $x^q-x$, then it is isomorphic to $(\mathbb{F}\cdot1+\mathbb{F}v,\lambda)$, where $v^2=0$, for some $\lambda\in\mathbb{F}$.
\end{lemma}
\begin{proof}
As in the proof of Lemma~\ref{Cap_3}, $\mathcal{A}$ is a $\mathbb{F}$-subalgebra of a tr-prime algebra $Q$ over a field $\mathbb{E}\supseteq\mathbb{F}$, where $\dim_\mathbb{E}Q=2$, i.e., $Q$ is any of the algebras listed in Proposition~\ref{2_dim}. Then, the subalgebra $\{x\in Q\mid (x^q-x)^2=0\}$ is either $\mathbb{F}$, $\mathbb{F}\oplus\mathbb{F}$ or $\mathbb{F}\cdot1+\mathbb{E}v$, where $v^2=0$ and $\mathbb{F}\subseteq\mathbb{E}$. The first two satisfy $x^q-x$. Since $\mathcal{A}$ does not satisfy $x^q-x$, then $\mathcal{A}$ is a unital subalgebra of $\mathbb{F}\cdot1+\mathbb{E}v$, say, $\mathcal{A}=\mathbb{F}\cdot1+\mathcal{W}$, where $\mathcal{W}\subseteq\mathbb{E}v$. The identity $t(x)^q-t(x)$ implies that the image of the trace is in $\mathbb{F}\cdot1$. If $w\in\mathrm{Ker}\,t|_{\mathcal{W}}$, then $\mathbb{F}w$ is a nilpotent tr-ideal, contradicting the tr-primality of $\mathcal{A}$. Therefore, $\mathrm{Ker}\,t|_{\mathcal{W}}=0$, and, in particular, $\dim_\mathbb{F}\mathcal{W}=1$. Hence, $\mathcal{A}\cong\mathbb{F}\cdot1+\mathbb{F}v$.
\end{proof}

\section{Polynomial identities of prime tr-algebras}
In this section, we shall assume that $\mathbb{F}$ is a finite field with $q$ elements. The next result is crucial and will be heavily used in the sequel.

\begin{proposition}\label{prop}
Let $\mathcal{A}$ be a finite-dimensional centrally closed prime $\Omega$-algebra over a finite field $\mathbb{F}$, and $S\subseteq\mathrm{Id}(\mathcal{A})$. Then, $\mathrm{Id}(\mathcal{A})=\langle S\rangle$ if and only if
\begin{enumerate}
\item The relatively free algebra $\mathbb{F}\langle X,t\rangle/\langle S\rangle$ is semiprime;
\item For each prime $\Omega$-algebra $\mathcal{B}$, with $S\subseteq\mathrm{Id}(\mathcal{B})$, there exists a subalgebra $\mathcal{S}\subseteq\mathcal{A}$ and a surjective algebra homomorphism $\mathcal{S}\to\mathcal{B}$.
\end{enumerate}
\end{proposition}
\begin{proof}
Assume that (i) and (ii) hold. Then, $\mathbb{F}\langle X,t\rangle/\langle S\rangle$ is a subdirect product of prime algebras; each satisfying all the polynomial identities of $\mathcal{A}$, by (ii). Hence, $\mathrm{Id}(\mathcal{A})\subseteq\langle S\rangle$, and, therefore, equality holds.

Conversely, assume that $\langle S\rangle=\mathrm{Id}(\mathcal{A})$. Then, it is well known that $\mathbb{F}\langle X,t\rangle/\mathrm{Id}(\mathcal{A})$ is semiprime. In addition, (ii) follows from Lemma~\ref{section}.
\end{proof}

It will be important for our purposes to classify all Capelli polynomials of a given order. When the relatively free algebra is commutative, we have an elementary description. We omit the following simple and tedious proof.
\begin{lemma}
Let $f$ be a polynomial that is Capelli of order $m$. Then, $f$ is a consequence of $[x,y]$ and 
\begin{equation}\label{alt}
\sum_{\sigma\in\mathcal{S}_{m}}(-1)^\sigma t(y_1x_{\sigma(1)})\cdots t(y_{m-1}x_{\sigma(m-1)})y_{m}x_{\sigma(m)}.
\end{equation}
\end{lemma}

Now, we compute the ideal of trace polynomial identities. This section is divided as follows. First, we deal with the algebra $(D_2,t_{(\alpha,\beta)})$ when $\alpha\ne\beta$ (4.1) and when $\alpha=\beta\ne0$ (4.2). In (4.3) we study the quadratic extension of the base field. Finally, we compute a basis of the trace polynomial identities of $(\mathbb{F}\cdot1+\mathbb{F}v,\lambda)$ in (4.4).

\subsection{} Let $\alpha$, $\beta\in\mathbb{F}$, with $\alpha\ne\beta$ and $\beta\ne0$, and consider the algebra $(D_2,t_{(\alpha,\beta)})$ over $\mathbb{F}$. We define the polynomials $r$ and $s$ via:
\begin{equation}\label{eq_rs}
\begin{split}
r(x_1,x_2)=t(x_1)t(x_2)-\beta t(x_1x_2)-\alpha t(x_1)x_2-\alpha t(x_2)x_1+\alpha(\beta+\alpha)x_1x_2,\\
s(x_1,x_2)=t(x_1)t(x_2)-\alpha t(x_1x_2)-\beta t(x_1)x_2-\beta t(x_2)x_1+\beta(\beta+\alpha)x_1x_2.
\end{split}
\end{equation}
Note that the image of $r$ evaluated on $(D_2,t_{(\alpha,\beta)})$ is contained in $\{0\}\oplus\mathbb{F}$, while the image of $s$ is contained in $\mathbb{F}\oplus\{0\}$. Hence,
\[r(x_1,x_2)s(x_3,x_4)
\in\mathrm{Id}(D_2,t_{(\alpha,\beta)}).
\]

Given $w_k=a_ke_1+b_ke_2\in\mathcal{A}(\mu_1,\mu_2,\nu_1,\nu_2)$, we have:
\begin{align*}
r(w_1,w_2)=(&a_1a_2\mu_1^2+(a_1b_2+b_1a_2)\mu_1\mu_2+b_1b_2\mu_2^2+a_1a_2(-\beta-2\alpha)\mu_1+\\&+(-\beta b_1b_2-\alpha b_1a_2-\alpha a_1b_2)\mu_2+\alpha(\beta+\alpha)a_1a_2)e_1+\\%
&(a_1a_2\nu_1^2+(a_1b_2+b_1a_2)\nu_1\nu_2+b_1b_2\nu_2^2+b_1b_2(-\beta-2\alpha)\nu_2+\\&+(-\beta a_1a_2-\alpha a_1b_2-\alpha b_1a_2)\nu_1+\alpha(\beta+\alpha)b_1b_2)e_2.
\end{align*}
The computation of $s(w_1,w_2)$ is analogous, as it is obtained by simply interchanging $\alpha$ and $\beta$.

We record some evaluations of the polynomials $r$ and $s$ on the algebra $\mathcal{A}(\mu_1,\mu_2,\nu_1,\nu_2)$:
\begin{equation}\label{ev_r}
\begin{aligned}
&r(e_1,e_1)=(\mu_1^2+(-\beta-2\alpha)\mu_1+\alpha(\beta+\alpha))e_1+(\nu_1^2-\beta\nu_1)e_2;\\%
&r(e_1,e_2)=(\mu_1\mu_2-\alpha\mu_2)e_1+(\nu_1\nu_2-\alpha\nu_1)e_2;\\%
&r(e_2,e_2)=(\mu_2^2-\beta\mu_2)e_1+(\nu_2^2+(-\beta-2\alpha)\nu_2+\alpha(\beta+\alpha))e_2,
\end{aligned}
\end{equation}
and, for $s$, we obtain:
\begin{equation}\label{ev_s}
\begin{aligned}
&s(e_1,e_1)=(\mu_1^2+(-\alpha-2\beta)\mu_1+\beta(\alpha+\beta))e_1+(\nu_1^2-\alpha\nu_1)e_2;\\%
&s(e_1,e_2)=(\mu_1\mu_2-\beta\mu_2)e_1+(\nu_1\nu_2-\beta\nu_1)e_2;\\%
&s(e_2,e_2)=(\mu_2^2-\alpha\mu_2)e_1+(\nu_2^2+(-\alpha-2\beta)\nu_2+\beta(\alpha+\beta))e_2.
\end{aligned}
\end{equation}

\begin{lemma}\label{lem_main}
If $\mathcal{A}=\mathcal{A}(\mu_1,\mu_2,\nu_1,\nu_2)$ is tr-prime and satisfies $t(1)-(\alpha+\beta)1$ and $r(x_1,x_2)s(x_3,x_4)$,
 then $\mathcal{A}(\mu_1,\mu_2,\nu_1,\nu_2)\cong(D_2,t_{(\alpha,\beta)})$.
\end{lemma}
\begin{proof}
From Lemma~\ref{rem}, we may assume that $\mu_2=\nu_2\ne0$. In addition, the identity $t(1)-(\alpha+\beta)1$ implies $\mu_1+\mu_2=\alpha+\beta$. Since $r(e_1,e_2)s(e_1,e_2)=0$, we get $\mu_2^2(\mu_1-\alpha)(\mu_1-\beta)=0$. Therefore, after interchanging the entries, either $\mu_1=\alpha$ or $\mu_1=\beta$. In any case, we obtain $\mathcal{A}\cong(D_2,t_{(\alpha,\beta)})$.
\end{proof}

Finally, we obtain a description of the trace polynomial identities of the algebra $(D_2,t_{(\alpha,\beta)})$, when the trace is non-degenerate and the base field is finite.
\begin{thm}\label{thm_alphabeta}
Let $\mathbb{F}$ be a finite field with $q$ elements, and let $\alpha$, $\beta\in\mathbb{F}$, with $\alpha\ne\beta$ and $\alpha,\beta\ne0$. Then, $\mathrm{Id}(D_2,t_{(\alpha,\beta)})$ follows from:
\begin{enumerate}
\renewcommand{\labelenumi}{(\roman{enumi})}
\item $r(x_1,x_2)s(x_3,x_4)$ (defined in \eqref{eq_rs}),
\item $t(1)-(\alpha+\beta)1$,
\item $x^q-x$, and
\item all Capelli polynomials of order $3$ (i.e., the polynomial \eqref{alt} with $m=3$).
\end{enumerate}
\end{thm}
\begin{proof}
Let $S$ be the set of all polynomial identities of the statement of the theorem. We shall apply Proposition~\ref{prop}. It is clear that any algebra satisfying $x^q-x$ is semiprime (indeed, it is semiprimitive). In particular, $\mathbb{F}\langle X,t\rangle/\langle S\rangle$ is semiprime.

Now, let $\mathcal{A}$ be a unital tr-prime algebra such that $S\subseteq\mathrm{Id}(\mathcal{A})$. By Lemma~\ref{Cap_3}, we obtain $\mathcal{A}\subseteq\mathbb{F}\oplus\mathbb{F}$. If $\dim_\mathbb{F}\mathcal{A}=1$, then the identity $t(1)-(\alpha+\beta)1$ implies that $\mathcal{A}$ is a subalgebra of $(D_2,t_{(\alpha,\beta)})$. If $\dim_\mathbb{F}\mathcal{A}=2$, then Lemma~\ref{lem_main} implies that $\mathcal{A}\cong(D_2,t_{(\alpha,\beta)})$. Therefore, we conclude from Proposition~\ref{prop} that $\mathrm{Id}(D_2,t_{(\alpha,\beta)})=\langle S\rangle$.
\end{proof}

\subsection{}
Let $\lambda\ne0$. We consider the algebra $(D_2,t_{(\lambda,\lambda)})$. The expression $y_2(x):=\lambda^{-1}t(x)-x$ is such that $y_2(a,b)=(b,a)$. Hence,
\begin{equation}\label{eq_lambda2}
h_2(x):=(\lambda^{-1}t(x))^2-\lambda^{-1}t(x^2)-2(\lambda^{-1}t(x)-x)x\in\mathrm{Id}(D_2,t_{(\lambda,\lambda)}).
\end{equation}
\begin{lemma}\label{prime_lambda2}
Let $\mathcal{A}=(D_2,t_{(\mu_1,\mu_2)})$ be a tr-prime algebra satisfying \eqref{eq_lambda2}. If $\mathrm{char}\,\mathbb{F}\ne2$, then $\mu_1=\mu_2=\lambda$. If $\mathrm{char}\,\mathbb{F}=2$ and $\mathcal{A}$ satisfies the identity $t(1)=0$, then $\mu_1=\mu_2=\lambda$.
%
%In other words, %\mathcal{A}\cong(D_2,t_{(\lambda,\lambda)})$.
\end{lemma}
\begin{proof}
Given $(a,b)\in\mathcal{A}$, the condition $h_2(a,b)=0$ gives us
\[
\lambda^{-2}(\mu_1a+\mu_2b)^2-\lambda^{-1}(\mu_1a^2+\mu_2b^2)+2a^2-2\lambda^{-1}(\mu_1a+\mu_2b)a=0.
\]
Now, assuming that $a$, $b\in\{0,1\}$, we get $a^2=a$ and $b^2=b$. Thus, we obtain a quadratic equation in the unknown $\lambda^{-1}(\mu_1a+\mu_2b)$, where the solution is
\begin{equation}\label{aux2}
\lambda^{-1}(\mu_1a+\mu_2b)\in\{1,2a\}.
\end{equation}
Now, taking $a=0$ and $b=1$ in \eqref{aux2} gives $\lambda^{-1}\mu_2\in\{1,0\}$. Since $\mathcal{A}$ is tr-prime, we may assume that $\mu_2\ne0$. Thus, $\mu_2=\lambda$. Now, evaluating $a=b=1$ in \eqref{aux2}, we get $\lambda^{-1}(\mu_1+\mu_2)\in\{1,2\}$; while $a=1$ and $b=0$ gives $\lambda^{-1}\mu_1\in\{1,2\}$.

Now, assume that $\mathrm{char}\,\mathbb{F}\ne2$. Since $\lambda^{-1}\mu_2=1$, we obtain $\lambda^{-1}\mu_1=1$ as well. Hence, $\mathcal{A}\cong(D_2,t_{(\lambda,\lambda)})$.

As a final step, assume that $\mathrm{char}\,\mathbb{F}=2$. Suppose that $\lambda^{-1}\mu_1=0$, i.e., $\mu_1=0$. Then $\mathcal{A}=(D_2,t_{(0,\lambda)})$. However, $t(1_\mathcal{A})=\lambda\ne0$, so it does not satisfy $t(1)=0$. Hence, once again, we obtain $\lambda^{-1}\mu_1=1$.
\end{proof}

As a consequence, we obtain the following result:
\begin{thm}\label{thm_lambda}
Let $\mathbb{F}$ be a finite field with $q$ elements, and let $\lambda\in\mathbb{F}\setminus\{0\}$. Then $\mathrm{Id}(D_2,t_{(\lambda,\lambda)})$ follows from:
\begin{enumerate}
\renewcommand{\labelenumi}{(\roman{enumi})}
\item $(\lambda^{-1}t(x))^2-\lambda^{-1}t(x^2)-2(\lambda^{-1}t(x)-x)x$,
%\item $t(1)-2\lambda$,
\item $x^q-x$, 
\item all Capelli polynomials of order $3$ (i.e., the polynomial \eqref{alt} with $m=3$),
\item $t(1)-2\lambda$.
\end{enumerate}
%In addition, if $\mathrm{char}\,\mathbb{F}=2$, then $\mathrm{Id}(D_2,t_{(\lambda,\lambda)})$ follows from all the above identities together with the identity $t(1)=0$.
\end{thm}
\begin{proof}
The argument is the same as in the proof of Theorem~\ref{thm_alphabeta}, except that Lemma~\ref{lem_main} is replaced by Lemma~\ref{prime_lambda2}.
\end{proof}

\subsection{} Let $\mathbb{F}$ be a finite field with $q$ elements, let $\mathbb{E}$ be a quadratic extension of $\mathbb{F}$, and consider $\mathbb{E}$ as an $\mathbb{F}$-algebra. Denote by $\varphi:\mathbb{E}\to\mathbb{E}$ the map defined by $\varphi(a)=a^q$. Recall that $\mathrm{Aut}(\mathbb{E}/\mathbb{F})=\langle\varphi\rangle$.

\begin{thm}\label{thmA}
Consider the $\mathbb{F}$-algebra $(\mathbb{E},t_\alpha)$, where $\mathbb{E}/\mathbb{F}$ is a quadratic extension, $\alpha\in\mathbb{E}$ and $t_\alpha(x)=\alpha x$, and denote $q=|\mathbb{F}|$. Then, $\mathrm{Id}(\mathbb{E},t_\alpha)$ follows from:
\begin{enumerate}
\renewcommand{\labelenumi}{(\roman{enumi})}
\item $m_\alpha(t(1))$, where $m_\alpha(x)\in\mathbb{F}[x]$ is the minimal polynomial of $\alpha$ over $\mathbb{F}$,
\item $t(x)-t(1)x$,
\item $x^{q^2}-x$, and
\item all Capelli polynomials of order $2$ (i.e., the polynomial \eqref{alt} with $m=2$).
\end{enumerate}
\end{thm}
\begin{proof}
Let $S$ be the set of polynomials of the statement. The polynomial identity $x^{q^2}-x$ implies that $\mathbb{F}\langle X,t\rangle/\langle S\rangle$ is semiprime. Let $(\mathcal{B},t)$ be a tr-prime algebra such that $S\subseteq\mathrm{Id}(\mathcal{B})$. Since $\mathcal{B}$ satisfies all Capelli polynomials of order $2$, it follows from Razmyslov Rank Theorem (Lemma~\ref{lemRaz3}) that $\dim_{C(\mathcal{B})}Q(\mathcal{B})=1$; i.e., $Q(\mathcal{B})=C(\mathcal{B})$ is a field extension of $\mathbb{F}$. Since $\mathcal{B}$ satisfies $x^{q^2}-x$, we see that, as ordinary algebras, $\mathcal{B}\subseteq\{x\in C(\mathcal{B})\mid x^{q^2}=x\}=\mathbb{E}$. If $\mathcal{B}=\mathbb{F}$, then the identity $t(x)-t(1)x$ implies that $\mathcal{B}$ is a tr-subalgebra of $(\mathbb{E},t_\alpha)$.

So, assume that $\mathcal{B}$ is a quadratic extension of $\mathbb{F}$. The first identity implies that $t(1_\mathcal{B})\in\{\alpha,\alpha^q\}$. In addition, the second identity implies that $(\mathcal{B},t)=(\mathbb{E},t_{\alpha})$ or $(\mathcal{B},t)=(\mathbb{E},t_{\alpha^q})$. Since both algebras are isomorphic, the result follows from Proposition~\ref{prop}.
\end{proof}

Now, we assume that we have a trace map of the kind $\mathbb{E}\to\mathbb{F}$. Recall that one has $\mathrm{Tr}_c(x)=\mathrm{Tr}_{\mathbb{E}/\mathbb{F}}(cx)$, for all $x\in\mathbb{E}$, for some $c\in\mathbb{E}\setminus\{0\}$. Recall that $(\mathbb{E},\mathrm{Tr}_c)\cong(\mathbb{E},\mathrm{Tr}_{c'})$ if and only if $c'\in\{c,c^q\}$.

\begin{thm}\label{thmB}
Let $\mathbb{E}/\mathbb{F}$ be a quadratic extension, where
$|\mathbb{F}|=q$, and let $\mathrm{Tr}_c:\mathbb{E}\longrightarrow\mathbb{F}$ be defined by $\mathrm{Tr}_c(x)=\operatorname{Tr}_{\mathbb{E}/\mathbb{F}}(cx)$ for every $x\in\mathbb{E}$, where $c\in\mathbb{E}\setminus\{0\}$. Put $
s=\operatorname{Tr}_{\mathbb{E}/\mathbb{F}}(c),
\ n=N_{\mathbb{E}/\mathbb{F}}(c)$. Then $\operatorname{Id}(\mathbb{E},\mathrm{Tr}_c)$ follows from:
\begin{enumerate}
\renewcommand{\labelenumi}{(\roman{enumi})}
\item $t(x)t(y)-s\,t(xy)+n(x-x^q)(y-y^q)$;
\item $t(1)-s$;
\item $t(x)^q-t(x)$;
\item $t(x)(1-x^{q^2-1})$;
\item $x^{q^2}-x$; and
\item all Capelli polynomials of order $3$ (i.e., the polynomial \eqref{alt} with $m=3$).
\end{enumerate}
\end{thm}

\begin{proof}
Let $S$ be the set of trace polynomials listed in the statement. We
shall apply Proposition~\ref{prop}.

We first verify that $
S\subseteq\operatorname{Id}(\mathbb{E},t_0)$. The identities $
t(1)-s,\ t(x)^q-t(x),\ x^{q^2}-x,\ t(x)(1-x^{q^2-1})$ are immediate. It is clear that all Capelli polynomials of
order $3$ vanish on $\mathbb{E}$, since
$\dim_{\mathbb{F}}\mathbb{E}=2$. Moreover, since $\mathrm{Tr}_c(x)=cx+c^qx^q$, we have
\begin{align*}
\mathrm{Tr}_c(x)\mathrm{Tr}_c(y)
={}&c^2xy+c^{q+1}xy^q+c^{q+1}x^qy+c^{2q}x^qy^q,
\end{align*}
whereas
\begin{align*}
s\,\mathrm{Tr}_c(xy)
&=(c+c^q)(cxy+c^qx^qy^q)\\
&=c^2xy+c^{q+1}xy+c^{q+1}x^qy^q+c^{2q}x^qy^q.
\end{align*}
Therefore,
\begin{align*}
\mathrm{Tr}_c(x)\mathrm{Tr}_c(y)-s\,\mathrm{Tr}_c(xy)
&=n\bigl(xy^q+x^qy-xy-x^qy^q\bigr)\\
&=-n(x-x^q)(y-y^q).
\end{align*}
Hence,
\[
t(x)t(y)-s\,t(xy)+n(x-x^q)(y-y^q)
\]
is a trace polynomial identity of $(\mathbb{E},\mathrm{Tr}_c)$. Thus, $S\subseteq\operatorname{Id}(\mathbb{E},\mathrm{Tr}_c)$.

% We also have $
% t_0(x)(1-x^{q^2-1})=0$
% for every $x\in\mathbb{E}$. Indeed, this is clear if $x=0$, while
% $x^{q^2-1}=1$ whenever $x\neq0$.

% We next observe that $(\mathbb{E},t_0)$ is centrally closed. Since
% $\mathbb{E}$ is a field, its only nonzero ideal is $\mathbb{E}$ itself,
% and every element of its centroid is given by multiplication by some
% $u\in\mathbb{E}$. If multiplication by $u$ commutes with $t_0$, then
% $t_0(ux)=u\,t_0(x)$ for every $x\in\mathbb{E}$. Since $t_0\neq0$, we may choose
% $x\in\mathbb{E}$ such that $t_0(x)\neq0$. It follows that
% $
% u\in\mathbb{F}$. Consequently, $
% C(\mathbb{E},t_0)=\mathbb{F}
% \ \text{and}\
% Q(\mathbb{E},t_0)=\mathbb{E}$.

The identity $x^{q^2}-x$ implies that the relatively free algebra $
\mathbb{F}\langle X,t\rangle/\langle S\rangle$ is semiprime. Now, let $(\mathcal{B},t_0)$ be a unital tr-prime algebra such that $S\subseteq\operatorname{Id}(\mathcal{B})$. By Lemma~\ref{Cap_3v}, the algebra $\mathcal{B}$ is isomorphic either
to an $\mathbb{F}$-subalgebra of $\mathbb{E}$ or to an
$\mathbb{F}$-subalgebra of $D_2(\mathbb{E})$.  First, note that we cannot have $\mathcal{B}\cong D_2$, since $D_2$ does not satisfy the identity $t(x)(1-x^{q^2-1})$.

It follows that $\mathcal{B}\cong\mathbb{F}\ \text{or}\ 
\mathcal{B}\cong\mathbb{E}$. Suppose first that $\mathcal{B}\cong\mathbb{F}$. Since $t_0$ is
$\mathbb{F}$-linear and $t_0(1)=s$, we have $
t(a)=sa$ for every $a\in\mathbb{F}$. On the other hand,
\[
t_0(a)
=\operatorname{Tr}_{\mathbb{E}/\mathbb{F}}(ca)
=a\operatorname{Tr}_{\mathbb{E}/\mathbb{F}}(c)
=sa.
\]
Hence, $(\mathcal{B},t)$ is isomorphic to the trace
subalgebra $\mathbb{F}\subseteq(\mathbb{E},\mathrm{Tr}_c)$.

Assume now that $\mathcal{B}\cong\mathbb{E}$. The identity $t(x)^q-t(x)$ forces $t_0(\mathcal{B})\subseteq\mathbb{F}$. Then there exists a unique
$d\in\mathbb{E}$ such that
\[
t_0(x)=\operatorname{Tr}_{\mathbb{E}/\mathbb{F}}(dx)
=dx+d^qx^q
\]
for every $x\in\mathbb{E}$. The identity $t(1)-s$ gives $d+d^q=s$. Repeating the computation performed above with $d$ in place of $c$,
we obtain
\[
t(x)t(y)-s\,t(xy)
=-N_{\mathbb{E}/\mathbb{F}}(d)(x-x^q)(y-y^q).
\]
Since $\mathcal{B}$ satisfies
\[
t(x)t(y)-s\,t(xy)+n(x-x^q)(y-y^q)=0,
\]
it follows that
\[
\bigl(n-N_{\mathbb{E}/\mathbb{F}}(d)\bigr)
(x-x^q)(y-y^q)=0
\]
for all $x,y\in\mathbb{E}$. Choosing $x,y\notin\mathbb{F}$, we obtain $N_{\mathbb{E}/\mathbb{F}}(d)=n$.

Thus, $c$ and $d$ have the same trace and the same norm over
$\mathbb{F}$. Hence, they are roots of the same polynomial
\[
X^2-sX+n\in\mathbb{F}[X].
\]
We have two possibilities, $d=c
\ \text{or}\ d=c^q$. In any case, one has $(\mathcal{B},t_0)\cong(\mathbb{E},\mathrm{Tr}_c)$. We have proved that every unital tr-prime algebra satisfying all the
polynomials in $S$ is isomorphic to a trace subalgebra of
$(\mathbb{E},t_0)$. Hence, all the hypotheses of
Proposition~\ref{prop} are satisfied, and we are done.
\end{proof}

\subsection{} Finally, we let $\mathcal{A}=(\mathbb{F}\cdot1+\mathbb{F}v,\lambda)$, where $\lambda\in\mathbb{F}$ and $v^2=0$. Recall that the trace map is defined by $t(1)=\lambda$ and $t(v)=1$.

\begin{lemma}\label{lemmaC}
The algebra $(\mathbb{F}\cdot1+\mathbb{F}v,\lambda)$ satisfies all the following polynomial identities:
\begin{enumerate}
\renewcommand{\labelenumi}{(\roman{enumi})}
\item $t(1)-\lambda$,
\item $t(x^q)-t(1)x^q$,
%\item $[x_1,x_2]$,
\item $(x^q-x)^2$,
\item $t(x)^q-t(x)$, and
\item all Capelli polynomials of order $3$ (i.e., the polynomial \eqref{alt} with $m=3$).
\end{enumerate}
In addition, if a tr-prime algebra $\mathcal{B}$ satisfies all the above polynomial identities, then it is a subalgebra of $(\mathbb{F}\cdot1+\mathbb{F}v,\lambda)$.
\end{lemma}
\begin{proof}
Let $S$ be the set of the polynomials of the statement. It is not hard to prove that $S\subseteq\mathrm{Id}(\mathbb{F}\cdot1+\mathbb{F}v,\lambda)$.

Let $\mathcal{B}$ be a tr-prime algebra with $S\subseteq\mathrm{Id}(\mathcal{B})$. Then, Lemma~\ref{Cap_3_weak} implies that $\mathcal{B}\subseteq D_2$ or $\mathcal{B}\cong(\mathbb{F}\cdot1+\mathbb{F}v,\lambda')$. The second case is done, since the identity $t(1)-\lambda$ forces $\lambda'=\lambda$. So, assume that $\mathcal{B}\subseteq D_2$. If $\dim_\mathbb{F}\mathcal{B}=1$, then there is nothing to do. If $\mathcal{B}\cong(D_2,t_{(\alpha,\beta)})$, then we obtain a contradiction, since $D_2$ does not satisfy $t(x^q)-t(1)x^q$.
\end{proof}

Now, to apply Proposition~\ref{prop}, we need to find a list of polynomials where the respective relatively free algebra is tr-semiprime.

\begin{lemma}\label{lemmaD}
Let $S$ be the set of all polynomials listed in Lemma~\ref{lemmaC}, together with $t(x-x^q)^{q-1}(x-x^q)-(x-x^q)$. Then $\mathbb{F}\langle X,t\rangle/\langle S\rangle$ is tr-semiprime.
\end{lemma}
\begin{proof}
Let $R=\mathbb{F}\langle X,t\rangle/\langle S\rangle$ and let $I$ be a nilpotent tr-ideal of $R$. We prove that $I=0$. For, take $f\in I$. Since $I$ is nilpotent, both $f$ and $t(f)$ are nilpotent. On the other hand, the identity $t(x)^q-t(x)=0$ gives us $t(f)^q=t(f)$. Now, since $t(f)$ is nilpotent, it follows that $t(f)=0$. Analogously, since
$f^q\in I$, $t(f^q)\in I$ is nilpotent, and the same identity gives $t(f^q)=0$. Hence we get $
t(f-f^q)=0$. Evaluating the identity
\[
t(x-x^q)^{q-1}(x-x^q)-(x-x^q)=0
\]
at $x=f$, we obtain $
-(f-f^q)=0$, i.e., $f=f^q$. Iterating this equality yields
$f=f^{q^m}$ for every $m\geq 1$. Since $f$ is nilpotent, we may choose $m$ such that
$f^{q^m}=0$. Hence $f=0$. Therefore every element of $I$ is zero, so $I=0$ and the proof follows.
\end{proof}

% \begin{example}
% Let $\lambda\in\mathbb{F}\setminus\{0\}$ and let $S$ be the polynomials listed in Lemma~\ref{lemmaC}. Define $\mathcal{A}=\mathbb{F}[\xi]/\langle\xi^2\rangle$ endowed with the trace map $t_\lambda(x)=\lambda x$. Then, it is not hard to see that $(\mathcal{A},t_\lambda)$ satisfies all of the polynomials in $S$; however, it does not satisfy $t(x)^q-t(x)$, nor $t(x-x^q)^{q-1}(x-x^q)-(x-x^q)$. Therefore, $t(x)^q-t(x)$ and $t(x-x^q)^{q-1}(x-x^q)-(x-x^q)$ are not consequences of $S$.
% \end{example}

As a consequence, we combine Proposition~\ref{prop}, together with Lemmas~\ref{lemmaC} and \ref{lemmaD}, to obtain a characterization of $\mathrm{Id}(\mathbb{F}\cdot1+\mathbb{F}v,\lambda)$.
\begin{thm}\label{thmC}
Let $\mathbb{F}$ be a finite field with $q$ elements, $\lambda\in\mathbb{F}$ and consider the algebra $(\mathbb{F}\cdot1+\mathbb{F}v,\lambda)$. Then, $\mathrm{Id}(\mathbb{F}\cdot1+\mathbb{F}v,\lambda)$ follows from:
\begin{enumerate}
\renewcommand{\labelenumi}{(\roman{enumi})}
\item $t(1)-\lambda$,
\item $t(x^q)-t(1)x^q$,
%\item $[x_1,x_2]$,
\item $(x^q-x)^2$,
\item $t(x)^q-t(x)$,
\item $t(x-x^q)^{q-1}(x-x^q)-(x-x^q)$, and
\item all Capelli polynomials of order $3$ (i.e., the polynomial \eqref{alt} with $m=3$).\qed
\end{enumerate}
\end{thm}

\section{Further discussion\label{sec}}
In this section, we discuss the classical approach to obtaining an explicit finite basis for a given finite algebra, as opposed to applying Proposition~\ref{prop}.

An $\Omega$-algebra is said to be \emph{subdirectly irreducible} if the intersection of all its nontrivial ideals is nontrivial. Equivalently, a subdirectly irreducible algebra cannot be represented as a nontrivial subdirect product of algebras. It is well known that every $\Omega$-algebra is a subdirect product of subdirectly irreducible $\Omega$-algebras. In particular, every algebra is PI-equivalent to a family of subdirectly irreducible algebras. Moreover, it is straightforward to see that every variety is generated by its finitely generated algebras. Combining these facts and specializing to our setting, we obtain:
\begin{lemma}
Every variety of trace algebras is generated by its finitely generated subdirectly irreducible trace algebras.\qed
\end{lemma}

As a consequence, we obtain the following well-known result, specialized to our setting:
\begin{lemma}\label{alternative_id}
Let $\mathcal{A}$ be a trace algebra and $S\subseteq\mathrm{Id}(\mathcal{A})$. If every finitely generated subdirectly irreducible trace algebra in $\mathrm{var}(S)$ is a section of $\mathcal{A}$, then $\mathrm{Id}(\mathcal{A})=\langle S\rangle$.\qed
\end{lemma}

Now, let $\mathcal{A}$ be a finite subdirectly irreducible trace algebra, and let $I$ be its unique minimal ideal, that is, the intersection of all nontrivial ideals. Then either $I^2=0$ or $I^2\neq0$. In the latter case, $\mathcal{A}$ is tr-prime, whereas in the former case, $\mathcal{A}$ is not semiprime as an ordinary algebra.

Now, if an algebra satisfies an identity of the kind $(x^a-x)^b$, with $a\in\mathbb{N}$ and $b\ge0$, then it is an algebraic algebra. By Kaplansky Theorem, a finitely generated algebraic algebra satisfying a polynomial identity is finite-dimensional. Thus, in our specific setting, it would suffice to consider finite subdirectly irreducible trace algebras. However, we will not need this fact here.

Now, let $S$ be a set of trace polynomial identities defined in any of Theorems~\ref{thm_alphabeta}, \ref{thm_lambda}, \ref{thmA}, or~\ref{thmB}. The identity of the form $x^{q'}-x$ automatically implies that every algebra in the variety $\mathrm{var}(S)$ is semiprime as an ordinary algebra. Hence, there is no difference between applying Lemma~\ref{alternative_id} and Proposition~\ref{prop} to characterize the respective T-ideal of trace polynomial identities, since in either case, the problem reduces to scrutinizing the tr-prime algebras satisfying $S$.

Now, let $S$ be the set of trace polynomial identities defined in Theorem~\ref{thmC}. As in the proof of Lemma~\ref{lemmaD}, one proves that any algebra satisfying all the identities of $S$ is tr-semiprime. By Lemma~\ref{lemmaC}, every tr-prime algebra satisfying $S$ is a trace subalgebra of $(\mathbb{F}\cdot1+\mathbb{F}v,\lambda)$. Hence, once again, there is no significant difference between applying Lemma~\ref{alternative_id} and Proposition~\ref{prop}.

\end{document}